\documentclass[12pt,a4paper]{article}

\usepackage[T1]{fontenc} 
\usepackage[utf8]{inputenc} 

\usepackage{amsmath}
\usepackage{amsthm}
\usepackage{hyperref}
\usepackage{cleveref}
\usepackage{amsfonts,enumitem}
\usepackage{amssymb}

\newcommand{\crite}{\mathrm{crit}_{\epsilon}\,}
\newcommand{\criteb}{\mathrm{crit}_{\bar{\epsilon}}\,}
\newcommand{\vcrite}{\mathrm{vcrit}_{\epsilon}\,}
\newcommand{\vcriteb}{\mathrm{vcrit}_{\bar{\epsilon}}\,}
\newcommand{\vcrit}{\mathrm{vcrit}\,}
\newcommand{\crit}{\mathrm{crit}\,}
\newcommand{\argmin}{\mathrm{argmin}\,}

\newcommand{\dist}{\mathrm{dist}}
\newcommand{\cl}{\mathrm{cl}\,}

\newcommand{\RR}{\mathbb{R}}

\newcommand{\N}{\mathbb{N}}

\newcommand{\partialc}{\partial^c}

\newtheorem{theorem}{Theorem}[section]

\newtheorem{theorem*}{Theorem}[section]
\newtheorem{lemma}{Lemma}[section]
\newtheorem{proposition}{Proposition}[section]
\newtheorem{corollary}{Corollary}[section]

\newtheorem{assumption}{Assumption}
\newtheorem{remark}{Remark}[section]

\renewenvironment{proof}[1][]{\noindent {\bf Proof #1:\;}}{\hfill $\Box$}

\providecommand{\keywords}[1]{\textbf{\textbf{Keywords. }} #1}

\newcommand{\R}{\mathbb{R}}

\title{Inexact subgradient methods for semialgebraic functions}

\author{Jérôme Bolte\footnote{Toulouse School of Economics, 
  Universit\'e Toulouse Capitole,
  Toulouse, France.} \and Tam Le \footnote{LPSM, Université Paris Cité.} \and Eric Moulines \footnote{CMAP, Ecole polytechnique, Paris and Mohamed Bin Zayed University of AI} \and Edouard Pauwels\footnote{Toulouse School of Economics, 
  Universit\'e Toulouse Capitole et IUF,
  Toulouse, France.}}

\begin{document}

\title{Inexact subgradient methods for semialgebraic functions
}


\author{Jérôme Bolte\footnote{Toulouse School of Economics, 
  Universit\'e Toulouse Capitole,
  Toulouse, France.} \and Tam Le \footnote{LPSM, Université Paris Cité.} \and Eric Moulines \footnote{CMAP, Ecole polytechnique, Paris and Mohamed Bin Zayed University of AI} \and Edouard Pauwels\footnote{Toulouse School of Economics, 
  Universit\'e Toulouse Capitole et IUF,
  Toulouse, France.}}

\maketitle

\begin{abstract} Motivated by the extensive application of approximate gradients in machine learning and optimization, we investigate inexact subgradient methods subject to persistent additive errors. Within a nonconvex semialgebraic framework, assuming boundedness or coercivity, we establish that the method yields iterates that eventually fluctuate near the critical set at a proximity characterized by an $O(\epsilon^\rho)$ distance, where $\epsilon$ denotes the magnitude of subgradient evaluation errors, and $\rho$ encapsulates geometric characteristics of the underlying problem. Our analysis comprehensively addresses both vanishing and constant step-size regimes. Notably, the latter regime inherently enlarges the fluctuation region, yet this enlargement remains on the order of $\epsilon^\rho$. In the convex scenario, employing a universal error bound applicable to coercive semialgebraic functions, we derive novel complexity results concerning averaged iterates. Additionally, our study produces auxiliary results of independent interest, including descent-type lemmas for nonsmooth nonconvex functions and an invariance principle governing the behavior of algorithmic sequences under small-step limits.

\end{abstract}

\keywords{Inexact subgradient, Clarke subdifferential, Nonsmooth nonconvex optimization, Path differentiable functions, First-order methods, Semialgebraic functions}

\textbf{AMS subject classifications:} {68Q25, 49J52, 49J53, 90C70}

\section{Introduction}
Let $f \colon \R^p \to \R$ be locally Lipschitz and consider the unconstrained global minimization problem
\begin{align*}
	\min_{x \in \R^p} f(x).
\end{align*}
An important tool for addressing such problems in  high-dimensional nonsmooth settings, are   subgradient algorithms, see  e.g.\cite{shor2012minimization,solodov1998error,davis2020stochastic,nguyen2023stochastic}. We focus here on the {\em inexact} or {\em biased subgradient method}, see \cite{solodov1998error}: 
\begin{equation*}
    x_{k+1} \in x_k - \alpha_k \left[\partial^c f(x_k)+ \bar B(0,\epsilon)\right], \: x_0\in \R^n\,,
\end{equation*}
where for all $k \in \N$, $\alpha_k > 0$ is a sequence of positive step sizes,  $\epsilon>0$ is an error or a bias level and $\partial^c$ denotes the Clarke subdifferential \cite{clarke1983optimization}. 

First, let us provide some motivations and insights into this method, particularly on its most distinctive features: inexact oracle evaluation and nonsmoothness.

 There are many sources of error in the evaluation of subgradients. They stem from numerical errors or gradient approximation techniques, such as complex derivation oracles \cite{griewank2003piggyback,makela2002survey,blondel2022efficient,bolte2023one} or low-accuracy calculations often used in neural network training \cite{de2017understanding}. The recent decade has seen a rapid proliferation of approximation methods, driven by the escalating dimensionality and structural complexity of optimization problems, such as large sums, composite functions, intrinsically complex neural network layers, differentiable programming, and federated learning frameworks. These developments have led to what can be broadly described as \textit{sketchy calculations}. Our approach leverages sketch calculus methods to achieve computational efficiency, reduce memory requirements, and effectively manage heterogeneous data. Examples of sketching techniques include mini-batching \cite{lecun1989backpropagation}, sparsification \cite{wangni2018gradient}, sophisticated compression methods \cite{mishchenko2019distributed}, and incremental computation strategies \cite{solodov1998error,nedic2010effect}. Federated and distributed learning frameworks, discussed in \cite{mishchenko2019distributed}, further exemplify scenarios involving inherently inexact subgradient computations. While stochastic optimization also extensively utilizes noisy gradient approximations due to random subsampling or Monte Carlo techniques \cite{Ermoliev1983StochasticQM,nemirovski2009robust,spall2012stochastic,ghadimi2013stochastic,ghadimi2016mini,lan2020algorithms}, we explicitly note that stochastic approximation techniques fall beyond the scope of this paper and warrant dedicated investigation.
 
  Nonsmooth (and inexact) subgradient methods have deep roots in optimization with the many works of Shor \cite{shor2012minimization}, Ermoliev \cite{Ermoliev1983StochasticQM}, Norkin \cite{Norkin1978nonlocal}, Nemirovskii \cite{nemirovskij1983problem}  and also the pioneering work of Solodov-Zavriev \cite{solodov1998error}.  
  Nonsmoothness is indeed  prevalent for both convex and nonconvex problems.  In the convex world, robustness questions naturally provide nonsmooth max problems \cite{ben2009robust}, regularization techniques resort to nonsmooth regularizers \cite{donoho2006compressed,tibshirani1996regression,combettes2011proximal,bach2011convex}, while bundle methods rely on complex polyhedral oracles, see e.g. \cite{makela2002survey}. Modern machine learning provides a wealth of applications involving nonsmoothness. While this aspect has always been witnessed in deep learning through standard building blocks such as ReLU activations and MaxPooling, it appears in more recent contexts. For instance, sorting procedures may be used to promote sparsity \cite{fedus2022switch}, while optimization layers \cite{amos2017optnet} may be used to refine learning abilities. 

These modern situations make the analysis of gradient methods more difficult. In addition to approximation errors and nonsmoothness, let us mention the use non-vanishing stepsizes schedules as  for instance in machine learning  contexts \cite{radford2018improving,loshchilov2017sgdr,cohen2020edge}. Studying non-vanishing steps is more challenging since they no longer mitigate oscillations and errors induced by nonsmoothness.

The main contributions of this article, emphasizing the most novel aspects, can be summarized as follows:
\begin{itemize}
\item[---] We develop an analytical framework to explicitly handle error terms, overcoming limitations of classical Lyapunov-based arguments typically used for subgradient analysis.
\item[---] We rigorously demonstrate that \textit{all iterates eventually fluctuate around the critical set}, with fluctuations confined within an $O(\epsilon^\rho)$ bound, directly related to the local geometric structure of the optimization landscape.
\item[---] We introduce several intermediate results that provide valuable insights for vanishing and constant step-size regimes. In particular, \Cref{lem:limitingSet} constitutes a general, independently significant result applicable to the ODE method, characterizing the limit points of sequences generated by small constant step-size recursions.
\item[---] Our analysis relies on weak and easily verifiable regularity assumptions on the cost function $f$. Specifically, our results hold broadly for local Lipschitz semialgebraic functions (or more generally, functions definable in an o-minimal structure), as detailed in \Cref{rem:localLipschitz} and \Cref{prop:coercive_semialg}.
\item[---] In the convex setting, we avoid restrictive assumptions such as compactness or strong error bounds. Instead, we simply require convexity, coercivity, and semialgebraicity.
\end{itemize}

Our work follows a substantial body of research dedicated to nonsmooth and inexact (sub)gradient methods. We provide below a concise, though not exhaustive, overview of this literature.

In the convex setting, there has been extensive research dating back to early studies within the Russian-Soviet school, see e.g.,  \cite{nedic2010effect,huh2014online} and references therein. 
 The inexact subgradient algorithm for Clarke regular nonconvex objective functions was first studied in the pioneering work of \cite{solodov1998error}. 
More recently, the analysis of stochastic gradient methods with biased errors has known several advances  \cite{tadic2011asymptotic,doucet2017asymptotic,ramaswamy2017analysis,nguyen2023stochastic}. These studies either focus on smooth objectives or rely on strong assumptions such as sharpness or metric regularity.  \cite{doucet2017asymptotic} provides results qualitatively similar  to ours, but exclusively for continuously differentiable targets; see also the recent review \cite{dieuleveut2023stochastic} for a broader perspective. For deterministic set-valued recursions, the seminal work \cite{benaim2012} provides fundamental qualitative insights. 

Our convergence analysis is based on the comparison of discrete-time sequences with the trajectories of continuous-time dynamical systems. This is a classical framework initiated by \cite{ljung1977analysis} and studied by \cite{benveniste2012adaptive,kushner2003stochastic,benaim1998recursive}, among others. This approach is particularly flexible, allowing to analyze stochastic approximation algorithms with constant or decreasing step sizes. 
Notable studies in the smooth case are \cite{benaim1998recursive,fort1999asymptotic}, \cite[Chapter~8]{kushner2003stochastic} and \cite{borkar2009stochastic}. This framework was  extended to situations where continuous-time dynamics are not determined by an ordinary differential equation but by upper semicontinuous differential inclusions, hence capturing nonsmooth algorithms. General results were first established in \cite{benaim2005stochastic}; see also \cite{benaim2012} for constant steps. The convergence of stochastic subgradient methods was then shown in \cite{majewski2018analysis} for Clarke regular functions, and in \cite{davis2020stochastic} for the broad class of definable functions.

Closely related to our work, \cite{josz2023global} analyzes the global stability of subgradient methods with constant stepsizes applied to locally Lipschitz, coercive, and definable functions. It is proved that iterates of methods approximated by subgradient trajectories  stabilize around critical points. Results specifically cover the subgradient method with momentum, stochastic subgradient methods with random reshuffling and momentum, and cyclic coordinate descent methods with random permutations.  
The key difficulty in extending this latter result to include subgradient errors lies in characterizing iterates behavior near critical levels, where standard descent mechanisms fail. We explain this in more detail in the reading keys of \Cref{section:readingKeys}.

\bigskip
Other studies worth mentioning, illustrating the scope of the dynamical system approach include \cite{Josz2023} which establishes a discrete notion of Lyapunov stability for local minima of a locally Lipschitz and semialgebraic function, and \cite{xiao2023convergence} dealing with the boundedness of the iterates in the stochastic setting. A parallel line of works also consider the stochastic approximation sequence as a Markovian process in order to study the limit of invariant measures as the steps goes to zero \cite[Section 8.4.3]{kushner2003stochastic}. This strategy proves particularly useful when studying constant step stochastic and nonsmooth algorithms \cite{roth2013stochastic,bianchi2019constant,bianchi2022convergence}. This literature is however quite distinct from our work and \cite{josz2023global} in terms of convergence notions and proof techniques.

\bigskip

\section{Preliminaries and statement of the main results}

\subsection{Notations.}\label{sec:notations}

For $f \colon \RR^p \to \RR$ locally Lipschitz, we denote its Clarke subdifferential by $\partialc f$,  and, for $\epsilon>0$, we define
\begin{align}
	\crite f&= \left\{x \ : \  \dist(0,\partialc f(x)) \leq \epsilon \right\} \nonumber, \\
	\vcrite f&= f(\crite f),
    \label{eq:critvalues}
\end{align}
 the $\epsilon$-critical set, which is closed and the set of {\em $\epsilon$-critical values}. When $l\notin \vcrite f$, it is called an {\em $\epsilon$-regular value}. 
We also denote by $\crit f$ the critical set and $\vcrit f$ the set of critical values. 
The set of minimizers is $\argmin  f$.  

For notational convenience, we write the algorithm as
\begin{equation}
    x_{k+1} = x_k - \alpha_k v_\epsilon(x_k),
	\label{eq:subgradVanishingStepSize}
\end{equation} 
where $v_\epsilon : \R^p \to \R$ represent the biased oracle, i.e. satisfies $\dist(v_\epsilon(x),\partialc f(x)) \allowbreak \leq \epsilon$ for all $x \in \R^p$.

Throughout the next sections, we use the shorthand $[a \leq g \leq b]$, for real numbers $a,b$ and function $g \colon \RR^p \to \RR$, to denote $\{ x \in \RR^p \ : \ a \leq g(x) \leq b\}$. We also use the same notations for $=$ and $<$, $>$.

We denote the Euclidean norm by $\|\cdot\|$. For $A \subset \RR^p$, we let $\|A\| = \sup_{a\in A} \|a\|$. For a subset, $A \subset \R^p$, $\overline{A}$ is its closure, and $A^c$ its complement. For $\epsilon > 0$, $\bar B(0,\epsilon)$ is the closed ball of center $0$ and radius $\epsilon$. We denote the distance of a point $x$ to a compact set $A \subset \R^p$ by $\dist (x, A) := \inf_{z \in A} \|x - z\|$. 
Given a set-valued map $Z:\R^p\rightrightarrows\R^q$, we denote the graph of $Z$ by $$\operatorname{graph}[Z] := \left\{(x,y)  \in \R^p \times \R^q \ : \ y \in Z(x) \right\}.$$

We call a function $f$ \emph{coercive} when for all $a \in \R$, the sublevel sets $[f \leq a]$ are compact, this is equivalent to $f(x) \to \infty$ when $\|x\| \to \infty$.

\subsection{Main results}

\paragraph{The nonconvex setting}
 We will work under the following assumption.
\begin{assumption}
    \label{ass:mainSemiAlgebraic}
    The function $f \colon \R^p\to\R$ is $L$-Lipschitz, lower-bounded, semialgebraic with $\crite f$ bounded for some $\epsilon>0$.
\end{assumption}

As stated in \Cref{lem:coercive}, \Cref{ass:mainSemiAlgebraic} ensures that $f$ is also coercive. Our first main result relates to vanishing step sizes.
\begin{theorem}[Convergence for biased subgradient method with vanishing step size]
	Under \Cref{ass:mainSemiAlgebraic}, there is $\bar{\epsilon} >0$, $C > 0$ $\rho > 0$ such that for any $\epsilon < \bar{\epsilon}$, $x_0 \in \R^p$, there is $\bar{\alpha} > 0$, 
	such that for any $(x_k)_{k \in \N}$ given by \eqref{eq:subgradVanishingStepSize} with $0<\alpha_k \leq \bar{\alpha}$ for all $k \in \N$,  $\alpha_k \to 0$ and $\sum_{k =0}^\infty \alpha_k = \infty$, we have 
	\begin{align*}
		\lim_{k \to \infty} \,\dist (f(x_k), \vcrite f ) &=0,\\
		\limsup_{k \to \infty} \,\dist (x_k, \crit f ) &\leq C \epsilon^{\rho}. 
	\end{align*}
	\label{th:vanishingStepSizeSemiAlgebraic}
\end{theorem}

Our second main result relates to small constant step sizes.
\begin{theorem}[Convergence for biased subgradient method with constant step size]
	Under \Cref{ass:mainSemiAlgebraic}, there is $\bar{\epsilon} >0$, $C > 0$ $\rho > 0$ such that for any $\epsilon < \bar{\epsilon}$, $x_0 \in \R^p$, and $(x_k(\alpha))_{k \in \N}$ given by \eqref{eq:subgradVanishingStepSize} with $\alpha_k = \alpha > 0$ 
	for all $k \in \N$, we have 
	\begin{align*}
		\lim_{\alpha \to 0^+} \,\limsup_{k \to \infty} \,\dist (f(x_k(\alpha)), \vcrite f ) &=0,\\
		\limsup_{\alpha \to 0^+}\, \limsup_{k \to \infty}\, \dist (x_k(\alpha), \crit f ) &\leq C \epsilon^{\rho}. 
	\end{align*}
	\label{th:constantStepSizeSemiAlgebraic}
\end{theorem}

Both these results are asymptotic regarding the iteration counter $k$. Obtaining non asymptotic convergence guarantees is open, even for the plain subgradient method $(\epsilon = 0)$.

\begin{remark}[Fluctuations with constant step sizes]In the previous theorem the limit with respect to $\alpha$, can be interpreted as follows: for any $c > 0$, there exists $\bar{\alpha}$, such that for all $0<\alpha \leq \bar{\alpha}$, the quantity $\dist (x_k(\alpha), \crit f )$ is of order $(C+c) \epsilon^{\rho}$ for sufficiently large $k \in \N$, where $C$ and $\rho$ are given in the theorems. In other words, the overall order of magnitude is $\epsilon^\rho$ with a constant depending on the step size, for small step sizes. Similar results were obtained by \cite{josz2023global} when $\epsilon = 0$.
\end{remark}

\begin{remark}[Local Lipschitz continuity]
    \label{rem:localLipschitz}
    As it is classical in optimization, the global Lipschitz continuity of $f$ in \Cref{ass:mainSemiAlgebraic} can be relaxed to mere local Lipschitz continuity provided that there is a mechanism ensuring that the generated sequence remains in a compact set. This is the case for our biased subgradient method. Indeed, considering $f$ as in \Cref{ass:mainSemiAlgebraic}, but only locally Lipschitz, the sequence remains bounded for small enough step sizes thanks to \Cref{lem:descentSequence} since the objective is coercive as stated in \Cref{lem:coercive} and $\vcrite f$ is upper bounded. Hence \Cref{th:vanishingStepSizeSemiAlgebraic} and \Cref{th:constantStepSizeSemiAlgebraic} actually holds if $f$ in \Cref{ass:mainSemiAlgebraic} is only locally Lipschitz rather than globally Lipschitz. We stick to global Lipschitz continuity of the objective $f$ for simplicity of the presentation. 
\end{remark}

\paragraph{The convex setting}
We complete these two results with an explicit estimate for the convex case. Our result is in the line of usual results, see, e.g., \cite{nedic2010effect,huh2014online}, but does not use any compactness or strong error bounds.

Let us recall \cite[Theorem 3]{bolte2017error}: for a coercive, convex, semialgebraic function $f \colon \R^p \to \R$, there exists $c>0$ and $a \in (0,1]$, such that
\begin{equation}
\label{eq:errorbound}
    \frac{c}{2}\left((f(x) - \min f)^a + (f(x) - \min f)\right) \geq \dist(x, \argmin f) \quad \forall x \in \R^p.
\end{equation}

\begin{theorem}[Biased subgradient complexity: convex case]
	Let $f \colon \allowbreak \R^p \to \R$ be $L$-Lipschitz, semialgebraic, convex and coercive. Then with $a \in (0,1]$ and $c>0$, as in  \eqref{eq:errorbound}, for any sequence generated by \eqref{eq:subgradVanishingStepSize}, we have
    \begin{align*}
	    \left( 2 - a - \epsilon c \right) \frac{\sum_{i=0}^k \alpha_i (f(x_i) - f^*)}{\sum_{i=0}^k \alpha_i}  & \leq  (1 - a)(\epsilon c)^{\frac{1}{1-a}} \\ & + \frac{\|x_0 - x_*\|^2 + (L+\epsilon)^2\sum_{i=0}^k  \alpha_i^2 }{\sum_{i=0}^k \alpha_i}.
    \end{align*}
    We understand for $a = 1$, the value $(1 - a)(\epsilon c)^{\frac{1}{1-a}}$ as $0$ if $\epsilon c < 1$ and undefined or infinite otherwise. 
    \label{th:convexSemiAlgebraic}
\end{theorem}

This allows to get convergence rates in value with various choices of step size. For example, if $\epsilon c < 1$, choosing $\alpha_i = \frac{1}{\sqrt{k + 1}}$ for $i = 0, \ldots, k,$ leads to 
\begin{align*}
    \min_{i=1,\ldots,k} f(x_k) - f^* &\leq  \frac{(1-a)(\epsilon c)^{\frac{1}{1-a}}}{2 - a - \epsilon c } +  \frac{\|x_0 - x_*\|^2 + (L+\epsilon)^2}{(2 - a - \epsilon c )\sqrt{k+1}}.
\end{align*}
Note that the bound is vacuous if for example $a = 1$ and $\epsilon \geq 1/c$ and provides effective guarantees only for small values of $\epsilon$. This is somewhat unavoidable: if $f$ is the absolute value and $\epsilon > 1$, one completely looses control of the resulting biased subgradient sequence. This is the so called ``low error'' setting.

\subsection{Reading keys and natural extensions}
\label{section:readingKeys}
For the two first theorems, we examine the continuous-time limit of the recursion \eqref{eq:subgradVanishingStepSize}, which leads to a differential inclusion given in \eqref{eq:biasDI}. A special Lyapunov mechanism is introduced in \Cref{sec:continuousTime}: Along the flow, the objective decreases only in certain regions of the space and may increase in other regions. This mechanism is used to describe the asymptotics of the system and provides explicit estimates in the Lipschitz coercive setting, under a Kurdyka-Lojasiewicz assumption for $f$ and a metric regularity assumption for $\partial^c f$. This leads to an explicit estimate of the distance to $\crit f$ for any positively invariant set $A$ such that $f(A) \subset \vcrite f$. Recall that $A$ is positively invariant for the differential inclusion \eqref{eq:biasDI} if for any $x \in A$ there is a solution to \eqref{eq:biasDI} starting at $x$ whose forward orbit is contained in $A$. The proof of \Cref{th:vanishingStepSizeSemiAlgebraic} and \Cref{th:constantStepSizeSemiAlgebraic} then consists of justifying the fact that relevant asymptotic sets fulfill this invariance property, which is done in \Cref{sec:discreteTimeAnalysis}. The main device is to use the small step limiting relation between the recursion \eqref{eq:subgradVanishingStepSize} and its continuous-time counterpart for which the analysis was made in \Cref{sec:continuousTime}.

Our arguments rely in part on the repulsivity of regular values, the complement of $\vcrite f$, as described in \Cref{sec:descentLemmas}, based on the continuous time analysis in \Cref{lem:lyapunovbiased}. A repulsivity mechanism of the same nature was instrumental in the proof of \cite[Theorem 1, Corollary 1]{josz2023global}, in a setting very close to ours, without errors. It turns out that, in the presence of errors, \Cref{lem:lyapunovbiased} does not provide a classical Lyapunov function, which should be decreasing, non-necessarily strictly, along all trajectories. Therefore, the presence of errors requires dedicated arguments and the repulsivity of regular values is not sufficient to conclude about \Cref{th:constantStepSizeSemiAlgebraic} and \Cref{th:vanishingStepSizeSemiAlgebraic}.

\Cref{th:convexSemiAlgebraic} is obtained by using standard Lyapunov functions in convex optimization and by applying an error bound condition that is automatically satisfied in the semialgebraic case. These arguments are given in \Cref{sec:convex} and are completely independent of the arguments for \Cref{th:vanishingStepSizeSemiAlgebraic} and \Cref{th:constantStepSizeSemiAlgebraic}, which make up the main technical part of this work.

These results could be readily extended in various ways. 
\begin{itemize}
	\item[---] It can be seen directly that the same results apply beyond semialgebraicity to any polynomially bounded o-minimal structure, the prototypical example being the structure of globally subanalytic sets, see \cite{van1998tame} for an overview.
	\item[---] It is also directly possible to obtain qualitatively similar results, without the convergence rate estimates, in any o-minimal structure, replacing the exponent estimates by definable functions \cite{van1998tame}.
	\item[---] While we consider the Clarke subdifferential, the main device used in the proof is the chain rule along Lipschitz curves, which is satisfied for the broader class of conservative gradient fields \cite{bolte2021conservative}. Examples are automatic differentiation oracles and generalized derivatives as in \cite{ermol1998stochastic}. It is known that conservative gradients may induce spurious stationary points. The result in \cite[Theorem 4.12]{bolte2022subgradient} ensures that up to removing a zero measure meagre set of possible initialization $x_0 \in \R^p$ and a finite set of step sizes $\alpha \in \R$, one does actually manipulate the subdifferential oracle.
	\item[---] We only consider the deterministic setting, it is possible to extend these results to the stochastic approximation setting, by adding zero mean stochastic perturbation terms which satisfy summability hypotheses for vanishing step sizes \cite{benaim2005stochastic} or by resorting to the study of convergence of stationary measures in the constant step size setting \cite{roth2013stochastic}.
\end{itemize}

\section{The continuous-time system and auxiliary results}
\label{sec:continuousTime}

Let $f \colon \RR^p \to \RR$ be locally Lipschitz. Consider for $\epsilon > 0$, the differential inclusion
\begin{equation}
    \dot{x}(t) \in -\partialc f(x(t)) + \bar B(0, \epsilon), 
    \label{eq:biasDI}
\end{equation}
for almost every $t \in [0,+\infty)$, where the solution is to be found among locally Lipschitz curves. Existence of solutions on maximal intervals follows from classical results, see \cite{aubin1983differential,Filippov1988DifferentialEW}. In particular, if $f$ is Lipschitz, we may consider solutions defined on $\R_+$. Equivalently, $x$ is solution to \eqref{eq:biasDI} if for almost every $t$, $\dist(\dot{x}(t), -\partialc f(x(t))) \leq \epsilon$.

\subsection{Asymptotics of the biased dynamics}
\label{sec:asymptoticsContinuoutDynamics}
Following Valadier and \cite{bolte2021conservative},  a locally Lipschitz function $f$ is called path-differentiable if for any Lipschitz curve $\gamma \colon \R \to \RR^p$, we have for almost every $t \in \R$,
\begin{align*}
	\frac{d}{dt} (f \circ \gamma)(t) = \left\langle v, \dot{\gamma}(t)\right\rangle,\qquad \forall v \in \partialc f(\gamma(t)).
\end{align*}
Semialgebraic functions are path-differentiable \cite{davis2020stochastic,bolte2021conservative}. Under path-differentia\-bility, we obtain the following results.

\begin{lemma}[Descent properties] \label{lem:lyapunovbiased} Let $f$ be locally Lipschitz and path-differentiable. Let $\epsilon > 0$ and $x : \R_+ \to \R^p$ be a solution to the differential inclusion \eqref{eq:biasDI} assumed to be defined on $\R_+$ (see remark below). Then,
\begin{enumerate}
    \item {\rm (Weak Lyapunov)} There is a measurable selection $v_x$ of $\partialc f(x(\cdot))$ such that,
    \begin{align*}
        f(x(t_2)) - f(x(t_1))&\leq - \int_{t_1}^{t_2} \|v_x(t)\| (\|v_x(t)\| - \epsilon) dt. & \forall  0 \leq t_1 \leq t_2.
    \end{align*}
    In particular, if $f(x(0)) \not \in \vcrite f$, there is $t>0$ such that $f(x(s)) < f(x(0))$ for all $0\leq s \leq t$.
		\item {\rm (Decrease)} 
						Let $l \not \in \cl\vcrite f$ be such that $f(x(0)) \leq l$ then for all $t > 0$, $f(x(t)) < l$. 
  		\item {\rm (Asymptotics)} Either $\|x(t)\| \to + \infty$ as $t\to+\infty$, or we have both \begin{align*}
		&\liminf_{t \to \infty} \dist(0,\partialc f(x(t))) \leq \epsilon\\
		& \lim_{t \to \infty} \dist(f(x(t)),\vcrite f) = 0.
\end{align*}
		\item {\rm (Quantitative estimates)} For any $0\leq a < b$ and $\delta > \epsilon$, set $T =  \frac{b- a}{\delta (\delta - \epsilon)}$ and assume that $f(x([0,T])) \subset [a,b]$, then there is $t \in [0,T]$ such that $\dist(0, \partialc f(x(t))) \leq \delta$. 
        \item {\rm ($\epsilon$-stationarity)} For any $a \in \R$ and $T>0$, if $f(x([0,T])) = \{a\}$ then $x([0,T]) \subset \crite f$.

\end{enumerate}
\end{lemma}
\begin{remark}{\rm 
	Actually \Cref{lem:lyapunovbiased} points 1, 2, 4, 5 hold on the interval of definition of the solution which does not need to be $\R_+$.     In our proofs, we will apply \Cref{lem:lyapunovbiased} to solutions defined on an interval $[0,T] \subset \R_+$.
	\label{rem:intervalDef}}
\end{remark}
\begin{proof}

1. Consider a measurable selection $v_x$ such that for almost every $t \geq 0$,

$$v_x(t) \in \argmin _{v \in \partialc f(x(t))} \|\dot{x}(t) + v\|.$$

Such a selection exists, see e.g. \cite[Theorem 18.13]{infinite}. Since $\dot{x}(t) \in - \partialc f(x(t)) + {\bar B(0,\epsilon)}$, then $\|\dot{x}(t) + v_x(t)\| \leq \epsilon$ for almost every $t \geq 0$. Then, by path-differentiability of $f$, we have for almost every $t \geq 0$,

\begin{align}
    \frac{d}{d t} (f \circ x)(t) & = \langle v_x(t), \dot{x}(t) \rangle  \nonumber\\
    & = \langle v_x(t), - v_x(t) + v_x(t) + \dot{x}(t)\rangle \nonumber\\
    & = - \|v_x(t)\|^2 + \langle  v_x(t),  v_x(t) + \dot{x}(t) \rangle \nonumber\\
    & \leq - \|v_x(t)\|(\|v_x(t)\| - \|\dot{x}(t) + v_x(t)\|) \nonumber\\
    & \leq  - \|v_x(t)\|(\|v_x(t)\| - \epsilon). 
		\label{eq:timeDerivativeObjective}
\end{align}

Integrating from $t_1$ to $t_2$ gives the desired inequality. If $f(x(0)) \not \in \vcrite f$, then $x(0) \not \in \crite f$, and since $\crite f$ is closed, there is a compact set $U$ with $x(0) \in \mathrm{int}\ U$ such that $U \cap \crite f = \emptyset$.
We may therefore choose $t > 0$ small enough  such that $x(s) \in U$ for all $s \in [0,t]$ and the result follows.

2. We distinguish two cases.

Case 1. Assume that $f(x(0)) \notin \cl\vcrite f$. Let us show that $f(x(t)) < f(x(0))$ for all $t > 0$. 

Since $(\cl \vcrite f)^c$ is open and $f$ is locally Lipschitz, there exists $t^* > 0$ small enough such that  $f(x(s)) \notin \cl\vcrite f$ for all $s \in [0, t^*]$. In this case, 1 gives us $f(x(s)) \leq  f(x(0)) - \int_{0}^{t^*} \|v_x(t)\| (\|v_x(t)\| - \epsilon\|) dt < f(x(0))$ for all $s \in (0,t^*]$.

Now we show that for all $t \geq t^*$, $f(x(t)) \leq f(x(t^*))$. Assume toward a contradiction that there exists $t \geq  t^*$ such that $f(x(t^*)) < f(x(t))$. Since $f(x(t^*)) \notin \cl\vcrite f$, we may chose $t$ so that $[f(x(t^*)), f(x(t))] \subset (\cl\vcrite f)^c$ by openness of the latter. Now, consider $t^- = \max \{s \ : \ s \in [t^*, t], f(x(s)) \leq f(x(t^*))\}$, and $t^+ = \min \{s \ : \ s \in [t^-, t], f(x(s)) \geq f(x(t)) \}$. By continuity of $f \circ x$, $t^-$ and $t^+$ are well defined with $t^+ \geq t^-$, and they satisfy for all $s \in [t^-, t^+]$, $f(x(s)) \in  [f(x(t^-)), f(x(t^+))] \subset (\cl \vcrite f)^c \subset (\vcrite f)^c$, as well as $f(x(t^-)) = f(x(t^*))$ and $f(x(t^+)) = f(x(t))$. In particular, $f(x(t^+)) > f(x(t^-))$ hence $t^+ > t^-$. Let $v_x$ be given by item 1. Then we have 

\begin{equation*}
    f(x(t^+)) - f(x(t^-)) \leq - \int_{t^-}^{t^+} \|v_x(s)\| (\|v_x(s)\| - \epsilon) ds < 0,
\end{equation*}
where the last inequality comes from $t^+ > t^-$ and $\|v_x(s)\| > \epsilon$ for almost every $s \in [t^-, t^+]$ since $f(x(s)) \notin \vcrite f$. This yields a contradiction. We have shown that for all $t \geq t^*$, $f(x(t)) \leq f(x(t^*))$.

Finally, for $t \in (0, t^*]$, $f(x(t)) < f(x(0))$, and for $t > t^*$, $f(x(t)) \leq f(x(t^*)) < f(x(0))$ hence the desired result under the assumption that $f(x(0)) \allowbreak \not \in \cl\vcrite f$.

Case 2. Now, assume that $f(x(0)) \in \cl\vcrite f$ and $f(x_0) \leq l$ for $l \not\in \cl\vcrite f$. In this case we actually have $f(x(0)) < l$. Since $(\cl\vcrite f)^c$ is open, there is $l' \not \in \cl\vcrite f$ such that $f(x(0)) < l' < l$. Then by continuity of $f \circ x$, either $f(x(t)) < l$ for all $t$ or there exists $t > 0$ such that $f(x(t))  = l'$ and $t$ can be chosen to be minimal since $[f \circ x = l']$ is closed and lower bounded. We have $f(x(s)) \leq l'$ for all $s \leq t$. Then by Case 1 shown previously, we have for all $s \geq t$, $f(x(s)) \leq l'$. Since $l' < l$, we have the desired result.

4. We now prove the fourth item, which does not depend on item 3 but is used to prove item~3. We have $a < b$.  Assume toward a contradiction that $\dist(0,\partial^c f(x(t))) > \delta$ for all $t \in [0,T]$. Since $s \to  \dist(0,\partial^c f(x(s)))$ is lower semi-continuous, there exists $\delta' > \delta$ such that $\dist(0,\partial^c f(x(t))) \geq  \delta'$ for all $t \in [0,T]$. In this case, item 1 gives us $f(x(T)) \leq f(x(0)) - T \delta' (\delta' - \epsilon)$, hence 
\begin{align*}
    f(x(T)) &\leq b - (b - a) \frac{\delta'(\delta' - \epsilon)}{\delta(\delta - \epsilon)} < b - (b - a) = a
\end{align*}
This is a contradiction as we assumed that $f(x(T))\geq a$.

3. Assume that $\|x(t)\|$ does not go to $+ \infty$, this means that the trajectory has accumulation points. So $f(x(t))$ also has finite accumulation values and in particular $f(x(t))$ does not diverge to $- \infty$ or $+\infty$.

--- Using item $1$, we have that $\liminf_{t \to \infty} \dist(0,\partialc f(x(t))) > \epsilon$ implies that $f(x(t)) \to - \infty$ as $t \to \infty$ and we obtain the first limit.

--- Denote by $I$ all the accumulation points of $f(x(t))$ as $t \to \infty$. $I$ is a nonempty  interval by continuity of $f \circ x$ and by  the fact that $f \circ x$ does not go to $+\infty$ or to $- \infty$. We distinguish two cases.

First if $I$ has empty interior, then $I = \{l\}$ and $f(x(t)) \to l$ for some $l \in \R$. 
Since the trajectory $x$ has accumulation points, we may find a sequence $(t_k)_{k \in \N}$ in $\R_+$ such that $t_k \to \infty$ and $x(t_k) \to \bar{x} \in \R^p$ as $k \to \infty$, and we have $f(\bar{x}) = l$. By item 4, for any $\delta > \epsilon$, set $a = l - \delta(\delta - \epsilon)$, $b = l + \delta(\delta - \epsilon)$ we have $f(x(t_k + \R_+)) \in [a,b]$ for $k$ sufficiently large and therefore, there exists $s_k \in [t_k,t_k + 2]$ such that $ \dist(0, \partialc f(x(s_k))) \leq \delta$. This can be repeated for smaller $\delta$ and we may find a sequence, $(s_k)_{k \in \N}$, such that for each $k$, $s_k \in [t_i,t_i + 2]$ for some $i \in \N$ and $\dist(0, \partialc f(x(s_k)) \to \epsilon$. The restricted trajectory  $\{x([t_i,t_{i}+2])\}_{i \in \N}$ remains in a bounded set by local Lischicity of $f$, hence local boundedness of its subdifferential, so up to a subsequence, we may assume that $x(s_k)$ converges to a point $\Tilde{x} \in \crite f$ with $f(\Tilde{x}) = l$. This shows that $l \in \vcrite f$.


Secondly, let us assume that $I$ has nonempty  interior.  Assume toward a contradiction that there exists $l \in \operatorname{int} I \cap (\cl\vcrite f)^c$. In this case, there is $u > 0$ such that $[l-u,l+u] \subset \operatorname{int} I \cap (\cl\vcrite f)^c$ and $t$ such that $f(x(t)) \in [l-u,l]$. By item 2, we have $f(x(s)) < l$ for all $s>t$, but this is contradictory with the fact that $l+u \in \operatorname{int} I$, because this implies that $t \mapsto f(x(t))$ has accumulation values strictly greater than $l+u$. So $I$ is an interval such that $\emptyset \neq \mathrm{int} I \subset \cl\vcrite f$ hence $I \subset \cl\vcrite f$. This means that $\sup_{v \in I} \dist(v,\vcrite f) = 0$, which is the desired result.

5. If $T > 0$ and $f \circ x$ is constant on $[0,T]$, by item 1, we necessarily have  $\dist(0, \partial^c f(x(t))) \leq \epsilon$ for almost every $t \in (0,T)$, and $x([0,T]) \in \crite f$ because $\crite f$ is closed and $x$ is continuous.



\end{proof}

\subsection{Estimates under the nonsmooth KL inequality and a metric subregularity condition}
\label{sec:continuousTimeEstimates}

We will obtain more precise estimates under the following assumption.

\begin{assumption} \label{ass:KL-MR}
 {\rm  $f$ is $L$-Lipschitz, path-differentiable,  the set $\vcrit f$ is nonempty finite, and there exists $\bar{\epsilon} \in (0,1)$ such that for any $0 \leq \epsilon \leq \bar{\epsilon}$, $\vcrite f$ is a finite union of segments.


Furthermore, there exists $c>0$ and $\theta \in [0,1)$, and $\beta > 0$ such that for all $x \in f^{-1}(\vcriteb f)$
	\begin{align}
					\dist(f(x), \vcrit f)^{\theta} &\leq c\, \dist(0, \partialc f(x)) \tag{KL} \label{eq:KL}\\
					\dist(x, \crit f) &\leq c\, \dist(0, \partialc f(x))^{\beta}\tag{MR}. \label{eq:MR}
	\end{align}
}	\label{ass:KLmetricReg}
\end{assumption}
Property \eqref{eq:KL} is some form of the nonsmooth Kurdyka-Lojasiewicz inequality \cite{bolte2007clarke}, while \eqref{eq:MR} is a form of metric sub-regularity of the subdifferential around the critical set, see  \cite[Proposition 3.1]{lee2022MR} for a general result, but also \cite{hesse2013nonconvex,aspelmeier2016local,russell2018quantitative,luke2020necessary} for concrete applications in optimization.

In our context, \Cref{ass:KL-MR} is essential to control trajectories of \eqref{eq:biasDI}. Some previous works on biased algorithms  relied as well on similar conditions. For instance, KL inequality for real analytic functions was used in \cite{doucet2017asymptotic} and metric regularity ($\beta = 1$) appears in \cite{nguyen2023stochastic}. \Cref{ass:KL-MR} is actually satisfied for all semialgebraic functions:

\begin{lemma}[Semialgebraicity  implies regularity]\label{lem:SA}
    If  \Cref{ass:mainSemiAlgebraic} is satisfied, then \Cref{ass:KL-MR} is also satisfied for some $\bar{\epsilon} > 0$. 
\end{lemma}
\begin{proof}
  According to Sard's theorem for semialgebraic functions \cite{bolte2007clarke}, $\vcrit f$ is finite. \Cref{ass:mainSemiAlgebraic} ensures that there exists $\bar{\epsilon}$ such that $\crite f$ is compact for every $0 \leq \epsilon \leq   \bar{\epsilon}$. For $0 \leq \epsilon \leq \bar{\epsilon}$ the sets $\vcrite f \subset \R$ are then compact and semialgebraic, i.e. they consist of a finite number of segments. 
    Furthermore, by \Cref{lem:coercive} (see next subsection), $f$ is coercive so $f^{-1}(\vcriteb f)$ is compact. 
    \eqref{eq:KL} follows from Kurdyka-Lojasiewicz inequality for nonsmooth semialgebraic functions \cite{bolte2007loja} and compactness. As for \eqref{eq:MR}, this is H\"older metric subregularity as given in \cite[Proposition 3.1]{lee2022MR} on the compact set $f^{-1}(\vcriteb f)$.
    
\end{proof}

We will therefore prove \Cref{th:vanishingStepSizeSemiAlgebraic} and \Cref{th:constantStepSizeSemiAlgebraic} under \Cref{ass:KLmetricReg}.

\begin{remark}[Beyond semialgebraicity ] {\rm 
        Semialgebrai\-ci\-ty in \Cref{ass:mainSemiAlgebraic} can be replaced by global subanalyticity and all results would hold true in the exact same form. This allows to include the logarithm and exponential function restricted to compact segments for example. More generally, our results hold provided that $f$ is definable in a polynomially bounded o-minimal structure \cite{coste2000introduction}, for which semialgebraic sets and globally subanalytic sets are the main examples. Our results could also be extended if in \Cref{ass:mainSemiAlgebraic} we assume that $f$ is definable in an o-minimal structure (not necessarily polynomially bounded) instead of being semialgebraic. Under this assumption, we would obtain the same results in  \Cref{th:vanishingStepSizeSemiAlgebraic} and \Cref{th:constantStepSizeSemiAlgebraic}, but the term of the form $C \epsilon^\rho$ would be replaced by a term of the form $e(\epsilon)$ for an abstract nonnegative increasing definable functions $e \colon \RR^+ \to \RR^+$ continuous at $0$ with value 0.
	\label{rem:semiAlgebraic}}
\end{remark}

The next lemma generalizes the following fact,  ``for a locally Lipschitz continuous semialgebraic $f$ and a subgradient curve $x$ ($\dot{x}(t) \in -\partialc f(x(t))$ for all $t \in \R_+$), the inclusion $f(x(\R_+)) \subset \vcrit f$, implies $0 \in \partialc f(x(t))$ for all $t \geq 0$''. Indeed, $\vcrit f$ is made of a finite number of singletons, and if there is $t> 0$ such that $\dist(0,\partialc f(x(t))) > 0$, this would hold locally and result in a strict decrease by path-differentiability. We would thus have  a time $t'>0$ such that $f(x(t')) \not \in \vcrit f$. This fact, corresponding to the case when $\epsilon=0$, can be extended to a general $\epsilon>0$. We remind the reader that $f(x) \in \vcrite f$ does not imply that $x \in \crite f$ so that it is not sufficient to control the distance between $\crite f$ and $\crit f$.

\begin{lemma}[Approximate stationarity of  near-critical  curves] \label{lem:approx_stationarity_curves} 
	Under \Cref{ass:KLmetricReg}, set $\rho = \frac{\beta}{\max\{\theta(\beta + 2),1\}}$. There exists $C >0$, such that for any $\epsilon \in [0, \bar{\epsilon}]$ and any solution curve $x : \R_+ \to \R^p$ such that $f(x(\R_+)) \subset \vcrite f$, we have for all $t \geq 0$,
	\begin{equation}
		\dist (x(t), \crit f ) \leq C\epsilon^{ \rho}.
	\end{equation}
	\label{lem:estimateSolutionToCrit}
\end{lemma}

\begin{proof} 
	Fix an arbitrary $\epsilon \leq \bar{\epsilon}$, and $x \colon \R_+ \to \R^p$ an arbitrary solution.
	Since $f(x(\R_+))$ is connected, $f(x(\R_+))$ is contained in a single connected component of $\vcrite f$ of the form $[a,b] \subset \R$. 
	
	For all $z \in \crite f$, $\dist(0, \partialc f(z)) \leq \epsilon$ implies $\dist(f(z), \vcrit f) \leq (c \epsilon)^{\frac{1}{\theta}}$ using \eqref{eq:KL}. Let $N \in \N$ such that $v_0 \leq v_1 \leq \cdots \leq v_{N+1}$ are the ordered critical values in $[a, b]$ to which we added $v_0 = a$ and $v_{N+1} = b$. 
 This defines $N+1$ segments which cover $[a,b]$, one of them has length at least $\frac{b-a}{N+1}$. Consequently, there is an open segment $(u,v)$ such that $v-u = \frac{b-a}{N+1}$ which does not contain any critical value. 
	Therefore, we may choose $z \in \crite f$ such that $f(z) = (u+v) / 2$ and we have 
	\begin{align}
		\dist(f(z), \vcrit f) \geq \frac{v-u}{2} = \frac{b-a}{2(N+1)}.
		\label{eq:pigeonHole}
	\end{align}
	Combining \eqref{eq:pigeonHole} with \eqref{eq:KL}, we obtain
	\begin{equation}
		b - a \leq 2(N+1) (c\epsilon)^{\frac{1}{\theta}} := K_1 \epsilon^{\frac{1}{\theta}}.
	    \label{eq:estimateLevelSlice}
	\end{equation}
	
	We fix an arbitrary $0<\alpha \leq 1$ and  $t \geq 0$. We set $\delta = 2 \epsilon^\alpha > \epsilon$, by point 4. of \Cref{lem:lyapunovbiased} and \eqref{eq:estimateLevelSlice}, there is $t'\geq t$, such that 
	\begin{align}
		\dist(0, \partialc f(x(t'))) &\leq 2\epsilon^\alpha \label{eq:deltaCrit}\\
		t'-t &\leq \frac{b - a}{(2\epsilon^\alpha - \epsilon)2\epsilon^\alpha } \leq \frac{K_1 \epsilon^{\frac{1}{\theta}}}{(2\epsilon^\alpha - \epsilon)2\epsilon^\alpha} \leq \frac{K_1}{2} \epsilon^{\frac{1}{\theta} - 2 \alpha}.
	    \label{eq:timemax}
	\end{align}
	It follows using the inclusion \eqref{eq:biasDI}, the fact that $f$ is $L$-Lipschitz so that its subgradient is bounded by $L$,  and \eqref{eq:timemax} that
	\begin{align}
		\|x(t') - x(t)\|  \leq \int_{t}^{t'} \|\dot{x}(s)\| ds \leq (t'-t) (L+\epsilon)  = O\left(\epsilon^{\frac{1}{\theta} - 2 \alpha} \right)
	    \label{eq:distanceTocritdelta}
	\end{align}
	Finally, using the estimates \eqref{eq:distanceTocritdelta}, \eqref{eq:deltaCrit} and \eqref{eq:MR},
	\begin{align}
	    \label{eq:equationtwotermsalpha}
		\dist(x(t), \crit f) & \leq \|x(t) - x(t')\| + \dist(x(t'), \crit f)  = O\left( \epsilon^{\frac{1}{\theta} - 2 \alpha}\right) + O\left( \epsilon^{\beta\alpha}\right)
	\end{align}
	Since $t \geq 0$ was arbitrary, the estimate \eqref{eq:equationtwotermsalpha} holds for all $t \geq 0$. Furthermore, since $0<\alpha \leq 1$ was arbitrary, we may choose $\alpha$ freely. We distinguish two cases.
	\begin{itemize}
		\item If $\theta(\beta+2) > 1$, then we choose $\alpha = \frac{1}{\theta (2 + \beta)}<1$, and the right-hand side in \eqref{eq:equationtwotermsalpha} is of the form $O\big(\epsilon^{\frac{\beta}{\theta (2+ \beta)}}\big)$.
		\item If $\theta(\beta+2) \leq 1$ then we may choose  $\alpha = 1$. 
			We have that $\epsilon^{\frac{1}{\theta} - 2} \leq \epsilon^{\beta}$
			hence the right-hand side in \eqref{eq:equationtwotermsalpha} is $O\left(\epsilon^{\beta}\right)$.
	\end{itemize}\end{proof}

\begin{remark}[On the initialization]
  {\rm   It may be puzzling not to see a condition on the initialization in \Cref{lem:approx_stationarity_curves}. This is hidden in the condition $f(x(\R_+)) \subset \vcrite f$ which enforces $x$ to start close enough to $\crit f$ so that $f(x(t))$ cannot leave $\vcrite f$ near $t=0$.}
\end{remark}
\begin{remark}[Power functions]
  {\rm  For a power function, $x \mapsto x^a$ on $\RR_+$, for $a > 1$, we have $\theta = 1- \frac{1}{a}$ and $\beta = \frac{1}{a - 1}$. In this case, we have
  \begin{align*}
  		\theta(\beta+2) = \frac{a-1}{a} \left(\frac{1}{a-1} + 2\right) = \frac{1}{a} + 2 \frac{a-1}{a} = 2 - \frac{1}{a} \in (1,2].
  \end{align*}
  Hence we have $\rho > \frac{\beta}{2}$. This corresponds to the estimate obtained by \cite{doucet2017asymptotic} for analytic functions (if $a$ is an integer). Indeed, in the notations of \cite[Proposition 8.2]{doucet2017asymptotic} we have $\theta = 1$ and $\beta = r_Q$ and the resulting estimate of \cite[Theorem 2.1 (iii)]{doucet2017asymptotic} corresponds to $\frac{\beta}{2}$ by combining estimates in \cite[Proposition 8.2]{doucet2017asymptotic} and \cite[Proposition 8.3]{doucet2017asymptotic}. In this univariate setting however, the correct estimate would be simply $\beta$ and we leave the question of the optimality of our estimate in the nonsmooth multivariate case for future research.
  }
\end{remark}

\Cref{lem:estimateSolutionToCrit} has the following direct consequence.
\begin{corollary}[Invariant sets and biased dynamics]
	Under \Cref{ass:KLmetricReg}, let $\epsilon \leq \bar{\epsilon}$, $S \subset \R^p$ be a positively invariant set, that is for any $z \in S$, there is $x \colon \R_+ \to \R^p$, solution to \eqref{eq:biasDI} such that $x(\R_+) \subset S$ and $x(0) = z$. If $f(S) \subset \vcrite f$, then for any $z \in S$, $\dist(z, \crit f) \leq C \epsilon^\rho$, where $C,\rho$ are given by \Cref{lem:estimateSolutionToCrit}. 
	\label{cor:invariantSet}
\end{corollary}
For the continuous time dynamics in \eqref{eq:biasDI}, \Cref{lem:lyapunovbiased} point 3 ensures that accumulation points of bounded solutions to the differential inclusion \eqref{eq:biasDI} correspond to objective values in $\vcrite f$. Furthermore, it is known that the set of such accumulation points forms an invariant set (see for example \cite[Theorem 3.6, Lemma 3.5]{benaim2005stochastic}). We obtain the following.
\begin{corollary}[Asymptotics of biased dynamics]
	Under \Cref{ass:KLmetricReg}, let $\epsilon \leq \bar{\epsilon}$, and $x \colon \R^+ \to \R$ be a bounded solution to the differential inclusion \eqref{eq:biasDI}. For any $z \in \R^p$, accumulation point of the trajectory, we have $\dist(z, \crit f) \leq C \epsilon^\rho$,  where $C,\rho$ are given by \Cref{lem:estimateSolutionToCrit}.
	\label{cor:continuousTimeLimit}
\end{corollary}

\subsection{Coercivity and boundedness of $\epsilon$ critical points}

Due to the importance of boundedness of curves in our results, we need to make a brief detour through coercivity properties and their link with the boundedness of $\crite f$. 
Let us recall first:
\begin{theorem}[Ekeland's variational principle] 
    \label{th:ekeland}
    Let $X$ be a complete metric space with distance $d$. Let $f \colon X \mapsto \RR \cup \{+ \infty\}$ be lower semi-continuous, bounded below and finite at least at one point. Fix $\epsilon > 0$ and $x_0 \in X$ such that $f(x_0) \leq \epsilon + \inf_x f(x)$. Then for any $\lambda >0$, there is $y_0 \in X$ such that
\begin{align*}
	f(y_0) \leq f(x_0) \qquad d(x_0,y_0) \leq  \lambda  \qquad  f(x) + \frac{\epsilon}{\lambda}d(x,y_0) > f(y_0),\, \forall  x \neq y_0.
\end{align*}
\end{theorem}
The conclusion tells us that $y_0$ is a strict global minimizer of
$\displaystyle
	g: x\rightarrow f(x) + \frac{\epsilon}{\lambda}d(x,y_0).$

In our framework, $\R^p$ is endowed with the Euclidean distance and  $f$ is locally Lipschitz, so using the sum rule with the Clarke subdifferential yields the following fact:
\begin{equation}\label{ekp}
0\in \partial^c g(y_0)\subset \partial^c f(y_0)+ \bar B\left(0,\frac{\epsilon}{\lambda}\right)
\end{equation}
This has the following consequences:
\begin{lemma}[Boundedness of the $\epsilon$-critical set implies coercivity]
	Assume that $f \colon \RR^p \to \RR$ is locally Lipschitz, then for any $x \in \RR^p$, $a \in (0,1)$, there is $y \in \RR^p$ such that
	\begin{align*}
		\|y\| &\geq \|x\| - \|x\|^a, &
		 \dist(0,\partialc f(y)) \|x\|^a \leq f(x) - \inf f.
	\end{align*}	
	In particular, if $f$ is lower bounded and $\criteb f$ is bounded for some $\bar{\epsilon}>0$, then $f$ is coercive, that is $f(x) \to +\infty$ as $\|x\|\to +\infty$.
	\label{lem:coercive}
\end{lemma}
\begin{proof}
	Fix $x \in \RR^p$, if $f(x) = \inf f$ there is nothing to prove. Choose $\epsilon = f(x) - \inf f$ (assumed finite, otherwise there is nothing to prove) and $\lambda = \|x\|^a$. By Ekeland's variational principle in \Cref{th:ekeland} there is $y \in \RR^p$ such that
	\begin{align*}
		\|x - y \| &\leq \|x\|^a,
	\end{align*}
	hence $\|y\| \geq \|x\| - \|x\|^a$. Moreover using \eqref{ekp},  i.e., $0\in \partial^c f(y)+ \bar B\left(0,\frac{\epsilon}{\|x\|^a}\right)$, we get $\dist(0,\partialc f(y)) \leq \frac{\epsilon }{\|x\|^a}$.  

If $\criteb f$ is bounded for some $\bar{\epsilon}>0$, and $f$ was not coercive, we could choose a sequence $\|x_k\|=k+1$ so that $f(x_k)-\inf f$ is bounded by some $M$, and obtain an unbounded sequence $y_k$ in $\criteb f$ for $k$ large enough.
\end{proof}

\begin{remark}[Coercivity and critical points]{\rm (a) Of course $\crit f$ bounded and nonempty does not imply coercivity, e.g.,  take $s \mapsto (1+s^2)^{-1}$. \\
(b) The fact that ``$\crite f$ bounded (for some $\epsilon>0$) implies $f$ coercive'' is a generalization of the classical result in convex analysis ``$\argmin  f$ bounded implies $f$ coercive when $f$ is convex''. }
\end{remark}

In general, coercivity alone does not imply $\criteb f$ bounded, even for semialgebraic functions, take for instance $s \mapsto \sqrt{|s|}$ on $\R \setminus [-1,1]$. The condition is however satisfied under sufficient growth, as showcased by the following result. Note the proposition below applies for example to objectives of the form of regularized risk $f(x) = \ell(x) + \frac{\lambda}{2} \|x\|^2$, where $\ell$ is bounded from below and semiagebraic, a widespread situation in machine learning.

\begin{proposition} \label{prop:coercive_semialg} Let $f : \R^p \to \R$ be a locally Lipschitz semialgebraic function. Assume there exists $\beta > 0$ such that $\liminf_{\|x\| \to \infty} \frac{f(x)}{\|x\|^\beta} > 0 $. Then 

1. If $\beta = 1$,  there exists $\Bar{\epsilon} > 0$ such that $\criteb f$ is bounded.

2. If $\beta > 1 $,  $\crite f$ is bounded for any $\epsilon > 0$.
\end{proposition}

\begin{proof} 1. We prove the case when $\beta = 1$. Assume toward a contradiction that for any $\epsilon > 0$, $\crite f$ is not bounded. Let $\Tilde{\epsilon} : \R_+ \to \R$ be a semialgebraic function such that $\Tilde{\epsilon}(t) \to 0$ as $t \to \infty$ and $\Tilde{\epsilon}(t) > 0$ for all $t \geq 0$. The set $\{(t, \Tilde{\epsilon}(t), x)  \ : \ \dist(0, \partialc f (x)) \leq \Tilde{\epsilon}(t) \}$ is semialgebraic and unbounded. Hence, by the curve selection Lemma, there exists a $C^1$ semialgebraic path $\gamma : \R_+ \to \R^p$ such that $\dist(0, \partialc f (\gamma(t)) \leq \Tilde{\epsilon}(t)$ for any $t \geq 0$ and $\|\gamma(t)\| \to \infty$ as $t \to \infty$.  By semialgebraicity, we may assume $\dot{\gamma}(t)$ does not vanish by considering only large $t$, and $\dot{\gamma}(t)/ \|\dot{\gamma}(t)\| \to v$ for some unit vector $v \in \R^p$. Up to a strictly increasing time reparameterization, we may assume $\|\dot{\gamma}(t)\| = 1$ at any $t \geq 0$. Note that this time reparameterization does imply that $\dot{\gamma}(t) \to v$ and does not change the fact that $\tilde{\epsilon}(t) \to 0$ as $t \to \infty$.

Set $\lambda :=  \liminf_{\|x\| \to \infty} \frac{f(x)}{\|x\|} > 0$ and $f_1(x) := f(x) - \frac{\lambda}{2}\|x\|$ and $N := \partialc \|\cdot\|$, the subdifferential of the norm. In particular, $f_1$ is coercive and semi-algebraic.

By applying path differentiability of $f_1$ along the $C^1$ curve $\gamma$, and using the sum rule for path differentiable functions \cite[Corollary 4]{bolte2021conservative}, we may write
\begin{align}
\label{eq:intermediateExampleCritBounded}
    f_1(\gamma(T)) - f_1(\gamma(0)) & = \int_0^T \langle \dot{\gamma}(t), - \frac{\lambda}{2} N(\gamma(t)) \rangle dt \nonumber\\
    &  + \int_0^T \langle \dot{\gamma}(t), \partialc f_1(\gamma(t)) + \frac{\lambda}{2} N(\gamma(t))\rangle dt \\
    & = - \frac{\lambda}{2} (\|\gamma(T)\| - \|\gamma(0)\|) + \int_0^T \langle \dot{\gamma}(t), \dist(0,\partialc f(\gamma(t)))\rangle dt \nonumber
\end{align}
Now, let us give some estimates. First, $\|\gamma(T)\| \geq \langle \gamma(T),v\rangle = \int_0^T \langle \dot{\gamma}(t),v\rangle dt$. And since $\dot{\gamma}(t) \to v$ we have $\|\gamma(T)\| \geq aT$ for some $a \in (0,1)$ for large values of $T$. Second, we have $ \int_0^T \langle \dot{\gamma}(t), \dist(0,\partialc f(\gamma(t)))\rangle dt = o(T)$, since its norm is bounded above by $\int_0^T \Tilde{\epsilon}(t) dt = o(T)$ as $\tilde{\epsilon(t)} \to 0$. Thus we deduce from \eqref{eq:intermediateExampleCritBounded} that for $T$ large enough, $f_1(\gamma(T)) = f_1(\gamma(0)) + \frac{\lambda}{2} \|\gamma(0)\| - (aT + b)  + o(T)$. Therefore $f(\gamma(T)) \to - \infty$ which contradicts the coercivity of $f_1$.

2. Let us prove the case where $\beta > 1$ using similar ideas. Assume toward a contradiction that there exists $\Bar{\epsilon} > 0$, such that $\criteb f$ is not bounded. We have a semialgebraic path $\gamma$ such that $\|\gamma(t)\| \to \infty$ and $\dist(0, \partialc f(\gamma(t))) \leq \Bar{\epsilon}$ for all $t \geq 0$. We then rely on arguments that are analogous to the proof of 1: $\gamma$ is chosen with unit speed, $\dot{\gamma}(t)/ \|\dot{\gamma}(t)\|$ converges, and we set $\lambda :=  \liminf_{\|x\| \to \infty} \frac{f(x)}{\|x\|^{\beta}}$, $f_\beta(x) := f(x) - \frac{\lambda}{2}\|x\|^{\beta}$ and $N := \partialc \|\cdot\|^{\beta}$. In particular, $f_\beta$ is coercive.
Applying similar arguments as in \eqref{eq:intermediateExampleCritBounded}, there exists $c > 0$ such that for large values of $T > 0$
\begin{equation*}
    f_\beta(\gamma(T)) \leq  f_\beta(\gamma(0)) - \frac{\lambda}{2} (\|\gamma(t)\|^{\beta} - \|\gamma(0)\|^{\beta}) + T \Bar{\epsilon} = -c T^\beta + O(T)
\end{equation*}
where we used that $\frac{\lambda}{2}\|\gamma(T)\|^{\beta} \geq cT^{\beta}$ for some $c > 0$. Again we have $f_\beta(\gamma(T)) \to -\infty$ as $T \to \infty$ which contradicts its coercivity.
\end{proof}

\section{Main consequences for the biased subgradient method}
\label{sec:discreteTimeAnalysis}

\subsection{Link between discrete and continuous-time}

In this section, we repeatedly use the connection between the discrete and continuous-time systems in the limit of small step sizes. This is a classical approach which we state  for a general set-valued map, see \cite[Lemma 21]{bolte2022long}. We use the following assumptions which guarantee the existence of solutions to the differential inclusion \cite[Chapter 2, Section 1, Theorem 3]{aubin1983differential}.
\begin{assumption}
	$Z$ is a set-valued map defined from $\R^p$ into the subsets of $\R^p$; it is nonempty convex valued and locally bounded with closed graph.
	\label{ass:setValuedMap}
\end{assumption}

\begin{lemma}[Approximate solutions to differential inclusions]
	\label{lem:discreteContinuous}
	Let $Z$ be as in \Cref{ass:setValuedMap}. For each $i\in \N$, let $T_i>0$ and assume that $T_i\to T$ as $i \to \infty$ for some $T>0$. For each $i\in\N$, let $\gamma_i\colon [0,T_i] \to \R^p$ be a Lipschitz curve. Assume that the sequence $(\gamma_i)_{i \in \N}$ converges to some bounded and Lipschitz curve $\gamma\colon[0,T]\to \R$, in the sense that $\sup_{t\in[0,\min(T_i,T)]}\|\gamma(t)-\gamma_i(t)\|\to 0$, and
 	\begin{equation}
		\label{eq:closetoclarke}
            \lim_{i\to \infty}\int_0^{T_i}\dist((\gamma_i(t),\dot{\gamma}_i(t)),\operatorname{graph}[Z])\,dt=0.
	\end{equation}
 Then $\dot{\gamma}(t)\in Z(\gamma(t))$ for almost every $t\in[0,T]$.
\end{lemma}
The following is classical and can be found for example in \cite{bolte2022long}. We will use it extensively to obtain discrete descent lemmas in the coming section.
\begin{lemma}[Affine interpolants of discrete dynamics and their limit curves]
    \label{rk:aff}
    Consider $T > 0$ and  sequences satisfying
    \begin{align*}
	x_{k+1} - x_k &\in \alpha_k    Z(x_k),& k =0,1,\ldots, K,
    \end{align*}
    where $K$ is such that $\sum_{k=0}^{K-1} \alpha_k \geq T$ and $0 <\alpha_k \leq \alpha$, $k = 0,\ldots,K-1$. Let $\gamma_{\alpha} \colon [0,T] \to \R^p$ be the interpolation of $(x_k)_{k = 0, \ldots, K}$ defined as follows: for all $ k =0,1,\ldots, K$, $\gamma_\alpha\left(\sum_{i=0}^{k-1} \alpha_k \right)= x_{k}$, and $\gamma_\alpha$ is affine between these points. The curve $\gamma_\alpha$ satisfies 
    \begin{align*}	 
        \int_0^{T}\dist((\gamma_\alpha(t),\dot{\gamma}_\alpha(t)),\operatorname{graph}[Z])\,dt &\leq \alpha T \sup_{0 \leq k\leq K-1}\|Z(x_k)\|.
    \end{align*}
    In particular, if $Z$ is bounded and we have a uniformly bounded family of such curves $\gamma_\alpha$, up to a subsequence, as $\alpha \to 0$, the curves converge uniformly to a solution of the underlying differential inclusion.
\end{lemma}
\begin{proof}
    For all $\alpha > 0$, the curve $\gamma_\alpha$ satisfies for any $t \in [0,T]$, 
\begin{align*}
	\dist((\gamma_\alpha(t),\dot{\gamma}_\alpha(t)),\operatorname{graph}[Z]) &\leq \alpha \sup_{0 \leq k\leq  K-1}\|Z(x_k)\|,
\end{align*}
as $\dot{\gamma_\alpha}(t) = Z(x_k)$ where $k$ is the closest point $x_k$ on the curve corresponding to time smaller than $t$.

When $Z$ is bounded by a constant $L$, one has a uniform bound since $\max_{0,\leq k\leq K}\|Z(x_k)\|\leq L.$ In that case, the curves $\gamma_\alpha$ are $L$-Lipschitz continuous and thus equicontinuous.\\
Whence, by letting $\alpha \to 0$, if the curves are bounded, Arzelà-Ascoli theorem applies and provides a subsequence of $\gamma_\alpha$ that uniformly converges to solutions of the continuous-time differential inclusion, thanks to \Cref{lem:discreteContinuous}.
\end{proof}

\subsection{Descent lemmas}
\label{sec:descentLemmas}

We now state some consequences of \Cref{lem:discreteContinuous} which apply to our setting.
\begin{lemma}[Quasi-descent Lemma]
	Let $f$ be locally Lipschitz and path-differentiable, $\epsilon >0$ and $l \not \in \vcrite f$. Then for any  $\eta>0$ and $M > 0$, there is $\bar{\alpha}>0$ such that for any $z \in \R^p$ satisfying $f(z)\leq l$, $\|z \| \leq M$, and for any sequence generated by \eqref{eq:subgradVanishingStepSize}, with $x_0 = z$ and $\alpha_k \leq \bar{\alpha}$, we have $f(x_k) \leq l+\eta$ for all $k < \inf\{i\in \N:\|x_i\| > M\}$. 
	
	In particular, if $f$ is coercive, there is $M> 0$, such that for $\alpha$ small enough, all sequences generated by \eqref{eq:subgradVanishingStepSize} initialized with $f(x_0) \leq l$ are bounded by $M$.
	\label{lem:descentSequence}
\end{lemma}
\begin{proof}
    Toward a contradiction, we assume that there are $\eta > 0$ and $M >0$ such that for any $\alpha > 0$, there exists $z = x_0$ satisfying $f(z) \leq l$ and a sequence $(\alpha_k)_{k \in \N}$ smaller than $\alpha$ such that there exists $K>0$ satisfying $f(x_K) > l + \eta$ and $\|x_k\| \leq M$ for all $k = 0,\ldots, K$. Observe that $(\alpha_k)_{k \in \N}$ depends on $ \alpha$, and denote by $L$ a Lipschitz constant of $f$ on the ball of radius $M$.

    We fix a time horizon $T < \eta / (L^2 + L\epsilon)$ and we may assume that $\alpha < T$.  We may then also find $k \in \N$ such that $f(x_k) \leq l$ and $f(x_i) > l$   for all $i = k+1, \ldots, K$. We have
     \begin{align*}
        l+\eta < f(x_{K}) \leq f(x_k) + L \sum_{j = k}^{K-1} \|x_{j+1} - x_j\| & \leq l +  L \sum_{j = k}^{K-1} \alpha_j (L + \epsilon) \\ & = l + (L^2 + L\epsilon) \sum_{j = k}^{K-1} \alpha_j . 
    \end{align*}
    We deduce that $\sum_{j = k}^{K-1} \alpha_j > T$.
    Let $K'\geq k$ be the smallest value such that $\sum_{j=k}^{K'} \alpha_j \geq T$, since $\alpha < T$ and by the argument above, we have $k < K' \leq K-1$. 
    We may then consider the truncated sequence $(x_i)_{i = k, \ldots, K'+1}$, which is nonempty whenever $\alpha<T$. Consider the affine interpolation of this truncated sequence from $i=k$ to $i=K'+1 \leq K$ and restricted to $[0,T]$ (which is possible because $\sum_{i=k}^{K'} \alpha_k \geq T$) as in \Cref{rk:aff}, call it $x_\alpha$, it is Lipschitz and satisfies:
    \begin{itemize}
        \item $\|x_\alpha(t)\| \leq M$ for all $t \in [0,T]$.
        \item $x_\alpha(0) = x_k$, $f(x_\alpha(0)) \leq l$ and $f(x_\alpha(t)) > l$ for some $t \in [0,\alpha]$.
        \item $l\leq f(x_\alpha(t)) \leq l + \eta + \alpha L(L+\epsilon)$ for all $t \in [\alpha,T]$.
        \item $\dist((x_{\alpha}(t),\dot{x}_\alpha(t)), \mathrm{graph} \left[-\partialc f + \bar{B}(0,\epsilon)\right]) \leq \alpha (L + \epsilon)$ and $\|\dot{x}_\alpha(t)\| \leq L + \epsilon$ for almost every $t$.
    \end{itemize}
    Thus, these trajectories are uniformly bounded and equicontinuous. Applying Arzelà-Ascoli theorem allows to obtain a converging subsequence as $\alpha \to 0$. \Cref{lem:discreteContinuous} ensures that the limit $x : [0,T] \to \R^p$ is a solution to \eqref{eq:biasDI} such that $f(x(0)) = l$ and $f(x([0,T])) \geq l$ which contradicts \Cref{lem:lyapunovbiased}.1 (see \Cref{rem:intervalDef} for its validity for $x$ defined on $[0,T]$).

    The last remark follows because if $f$ is coercive and locally Lipschitz, then it is Lipschitz on the (compact) sublevel set $l + \eta$, say with constant $L$. For a fixed step size threshold $\alpha_0>0$, choose $M_0 = \max\{\|x\| + \alpha_0 (L+\epsilon),\, \mathrm{s.t.}\,  f(x) \leq L + \eta\}$ and denote by $\tilde{L}$ a Lipschitz constant of $f$ on the ball or radius $M_0$ centered at zero. By coercivity, we may choose $M>0$ large enough such that $\inf_{\|x\| \geq M} f(x) > l + \eta + \alpha_0 \tilde{L}(L+\epsilon)$. We apply the lemma for this choice of $M$, reducing the resulting $\bar{\alpha}$ if it is bigger than $\alpha_0$. Fix any admissible sequence $(x_k)_{k\in\N}$ and $k$ such that $f(x_k) \leq l + \eta$. It holds that $\|x_{k+1} - x_k\| \leq \alpha_0 (L+\epsilon)$ and $\|x_k\| + \alpha_0 (L+\epsilon)\leq M_0$ so that both $x_k$ and $x_{k+1}$ are contained in the ball of radius $M_0$. We deduce that $f(x_{k+1}) \leq l+\eta+\alpha_0 \tilde{L}(L+ \epsilon)$ so $\|x_{k+1}\| \leq M$. We deduce that $k+1 < \inf\{i\in \N:\|x_i\| > M\}$ and the main statement of the lemma ensures that, $f(x_{k+1}) \leq l + \eta$. By induction this holds true for all $k \in \N$.
\end{proof}

\begin{lemma}[$\epsilon$-regular values are repulsive]
	Let $f \colon \RR^p \to \RR$ be a locally Lipschitz and path-differentiable function, $\epsilon > 0$, $l \not \in  \vcrite f$, and $\eta$ $:=$ $ \dist(l,\vcrite f) / 16 > 0$. For any $M > 0$, there is $\bar{\alpha}>0$ such that for any sequence generated by \eqref{eq:subgradVanishingStepSize}, with $\alpha_k \leq \bar{\alpha}$ for all $k \in \N$, and $\sum_{k \in \N} \alpha_k = + \infty$, either $\{i \in \N,\, \|x_i\| > M\}$ is nonempty  or $\liminf_{k \to \infty}f(x_k) > l+2\eta$ or $\limsup_{k \to \infty} f(x_k) < l- 2\eta$. 

        In particular under \Cref{ass:mainSemiAlgebraic}, there is $\bar{\alpha}>0$ such that regardless of the initial condition for any sequence generated by \eqref{eq:subgradVanishingStepSize}, with $\alpha_k \leq \bar{\alpha}$, and $\sum_{k \in \N} \alpha_k = + \infty$ for all $k \in \N$, $\liminf_{k \to \infty}f(x_k) > l+2\eta$ or $\limsup_{k \to \infty} f(x_k) \allowbreak < l- 2\eta$.
	\label{lem:descentSequence2}
\end{lemma}
\begin{proof}
    Fix $M > 0$ and denote by $L$ a Lipschitz constant of $f$ on the ball of radius $M$.
    By definition of $\eta$, we have that $[l-8\eta,l+8\eta] \subset (\vcrite f)^c$. Set $\delta := \min \{ \|v\| \ : v \in \partialc f(x),\, x \in \R^p,\, \|x\| \leq M,\,  |f(x)-l|\leq 8\eta \} > \epsilon$ and $T \geq 12\frac{\eta}{\delta(\delta - \epsilon)}$. Any solution $x : [0,T] \to \R^p$ to \eqref{eq:biasDI} bounded by $M$ such that $f(x(0))\leq l+4\eta$ satisfies $f(x(T)) \leq l - 8\eta$ by \Cref{lem:lyapunovbiased}.1. We first construct a thresholds step size $\bar{\alpha}>0$ which satisfies three desirable properties.

	1. There is $\bar{\alpha} > 0$, such that any sequence generated by \eqref{eq:subgradVanishingStepSize}, bounded by $M$, with $f(x_0)\leq l+4\eta$ satisfies $f(x_k) \leq l - 4\eta$ for some $k \in \N$. Indeed, if this was not the case, using \Cref{lem:discreteContinuous} and the fact that $\sum_{k \in \N} \alpha_k = + \infty$, we could construct a solution curve to the differential inclusion $x\colon \RR^+ \to \RR$ satisfying $f(x(0))\leq l+4\eta$ and $f(x(t)) \geq l- 4\eta$ for all $t \geq 0$. In particular $f(x(T)) \geq l - 4\eta$, which contradicts the previous claim stating that $f(x(T)) \leq l - 8\eta$ for any such curve. 
	
	2. We may assume that for any sequence bounded by $M$ such that  $\limsup_{k \to \infty} f(x_k) \geq l - 2\eta$ and $\liminf_{k\to \infty} f(x_k) \leq l+ 2\eta$, $(f(x_k))_{k \in \N}$  takes infinitely many value in $[l-4\eta,l+4\eta]$. Indeed, for all $k$, we have $\|x_k\| \leq M$ and $\|x_{k+1}\| \leq M$, so that 
	$$|f(x_{k+1})-f(x_k)|\leq L \|x_{k+1}-x_k\|\leq \bar \alpha L(L + \epsilon),$$
	 shrinking $\bar \alpha$ if necessary that this gap is less than $\eta$.

	3. We may reduce $\bar{\alpha}$ further, as given by \Cref{lem:descentSequence}, so that any sequence initialized with $f(x_0) \leq l- 4\eta$ and bounded by $M$ satisfies $f(x_k) \leq l - 3\eta$ for all $k \in \N$. The chosen $\bar{\alpha}$ satisfies the required properties.

	Let us summarize the properties of the resulting $\bar{\alpha}$. Note that the properties claimed above can be shifted by initializing a sequence at an arbitrary $K \in \N$ and by considering $k \geq K$. Below, $(x_k)_{k \in \N}$ is any sequence generated by \eqref{eq:subgradVanishingStepSize} with $\alpha_k \leq \bar{\alpha}$ for all $k \in \N$, bounded by $M$, and $K \in \N$ is arbitrary.
	\begin{enumerate}
		\item If  $f(x_K) \leq l+4\eta$ , there is $k \geq K$ such that $f(x_k) \leq l- 4\eta$.
	   \item If $\limsup_{k \to \infty} f(x_k) \geq l - 2\eta$ and $\liminf_{k \to \infty} f(x_k) \leq l+ 2\eta$, then there is $k \in \N$ such that $f(x_k) \in [l-4\eta,l+4\eta]$.
		\item If $f(x_K) \leq l- 4\eta$ then $f(x_k) \leq l - 3\eta$ for all $k \geq K$.
	\end{enumerate}

	For the choice of $\bar{\alpha}$ as above, assume toward a contradiction that there exists a sequence generated by \eqref{eq:subgradVanishingStepSize}, bounded by $M$, with $\alpha_k \leq \bar{\alpha}$ for all $k \in \N$, satisfying $\limsup_{k \to \infty} f(x_k) \geq l - 2\eta$ and $\liminf_{k \to \infty} f(x_k) \leq l+ 2\eta$. There would be $K \in \N$ such that $f(x_K) \in [l-4\eta,l+4\eta]$ by 2 and we deduce by 1 that there is $k \geq K$ such that $f(x_k) \leq l - 4\eta$. This implies by 3 that $\limsup_{k \to \infty} f(x_k) \leq l - 3\eta$ which is contradictory. 

    The last comment follows because by \Cref{lem:descentSequence} reducing $\bar{\alpha}$ if necessary, there is $M> 0$ such that all sequences generated by \eqref{eq:subgradVanishingStepSize} with $f(x_0) \leq l + 3 \eta$ and $\alpha_k \leq \bar{\alpha}$ for all $k \in \N$ are bounded by $M$. If $\liminf_{k \to \infty}f(x_k) \leq l+2\eta$, then there is $K>0$ such that $f(x_K)\leq l + 3 \eta$ and by considering $\tilde{x}_k = x_{k+K}$, bounded by $M$, the result follows.
\end{proof}

Under \Cref{ass:mainSemiAlgebraic}, $f$ is coercive and $\vcrite f$ is closed. Combining the two previous lemma we obtain
\begin{corollary}[Large-horizon descent Lemma]
Under \Cref{ass:mainSemiAlgebraic}, fix $x_0$ such that $f(x_0) \not \in \vcrite f$ and $\eta>0$.
	There is $\bar{\alpha}>0$ such that any sequence generated by \eqref{eq:subgradVanishingStepSize} with $\alpha_k \leq \bar{\alpha}$ and $\sum_{k\in \N} \alpha_k = +\infty$ we have 
	\begin{itemize}
	\item[(i)] $f(x_k) \leq f(x_0)+\eta$ for all $k \in \N$, 
	\item[(ii)] there is $\kappa$  such that 
	$$f(x_k)\leq c+\eta, \:\: \forall k\geq \kappa$$
	where $c$ is the first $\epsilon$-critical value below $f(x_0)$.
	\end{itemize}
\end{corollary}

Note that if $(1+\rho)\eta<f(x_0)-c$ for some $\rho>0$, the above implies 
$$f(x_k)\leq f(x_0)-\rho\eta \mbox{ for $k\geq\kappa$},$$
where $\kappa$ is unknown.

\subsection{Vanishing step sizes}
For vanishing step sizes, the work of Benaim-Hofbauer-Sorin \cite{benaim2005stochastic} ensures that the set of accumulation points is invariant. We additionally show that it has accumulation values in $\vcrite f$. The main result stated in \Cref{th:vanishingStepSizeSemiAlgebraic} is a consequence of the following result and \Cref{lem:SA}.
\begin{theorem}[Convergence for biased subgradient with vanishing step]
	Under \Cref{ass:KLmetricReg}, let $C, \rho$ be given by \Cref{lem:estimateSolutionToCrit} and $\epsilon < \bar{\epsilon}$. Assume that $(x_k)_{k \in \N}$, is given by \eqref{eq:subgradVanishingStepSize} with  $\alpha_k \to 0$ and $\sum_{k =0}^\infty \alpha_k = \infty$.
	 If $\sup_{k \in \N} \|x_k\| < \infty$,  
	\begin{align*}
					\lim_{k \to \infty} \dist (f(x_k), \vcrite f ) &=0,\\
			       \limsup_{k \to \infty} \dist (x_k, \crit f ) &\leq C\epsilon^{\rho}.
	\end{align*}
        The boundedness condition always holds for small enough step sizes if $f$ is coercive.
	\label{th:vanishingStepSize}
\end{theorem}
\begin{proof}
    By assumption, there is $M > 0$ such that for all $k \in \N$, $\|x_k\| \leq M$.
    Denote by $L_f$ the set of limit points of $f(x_k)$, it is an interval since $\alpha_k \to 0$. We first prove the first statement which is equivalent to the fact that $L_f$ does not contain any value $l \not \in \vcrite f$. Suppose toward a contradiction that this is the case, then there exists a value $l \in (\vcrite f)^c \cap L_f$. For such a value $l$, let $\eta > 0$ be given by \Cref{lem:descentSequence2} ($\vcrite f$ is closed by \Cref{ass:KL-MR}). Recall that $(\alpha_k)_{k \in \N}$ is vanishing so that either $\liminf_{k \to \infty} f(x_k) > l + 2\eta$ or $\limsup_{k \to \infty} f(x_k) < l - 2 \eta$, which is in contradiction with $l \in L_f$.

	Denote by $L$ the set of accumulation points of the sequence $(x_k)_{k \in \N}$. By the previous statement we have $f(L) \subset \vcrite f$. By \cite[Theorem 4.2 and 4.3]{benaim2005stochastic}, $L$ is an internally chain transitive set. In particular, $L$ is invariant for the inclusion \eqref{eq:biasDI} by \cite[Lemma 3.5]{benaim2005stochastic}. This means that for any $\bar{x} \in L$, there is a solution $x$ to \eqref{eq:biasDI}, such that $x(0) = \bar{x}$ and $x(\R) \in L$. The second statement follows by \Cref{cor:invariantSet}. 

    The last comment on boundedness follows from \Cref{lem:descentSequence}.
\end{proof}

\subsection{Constant step sizes}
We start this section with a general result on constant step size discretization which is of independent interest.
\begin{lemma}[Limits of accumulation points are invariant]
	Let $M>0$ be a bound and $x_0 \in \R^p$ be fixed.
	Under \Cref{ass:setValuedMap}, for any $s>0$, denote by $L(s)$ the set of accumulation points of sequences satisfying $x_{k+1}(s) = x_k(s) +  s h(x_k(s))$ for each $k \in \N$, where $h$ is a selection in $Z$. Assume that $L(s)$ is nonempty  and bounded by $M$ for all $s$ small enough. Set 
	\begin{align*}
		L = \bigcap_{\alpha > 0} \mathrm{cl} \left\{\bigcup_{0<s\leq \alpha} L(s)\right\}.
	\end{align*}
	Then the set $L$ is nonempty and invariant with respect to the flow induced by $Z$ in the sense that for any $\bar{x} \in L$ there exists a solution to the differential inclusion $\dot{x}(t) \in Z(x(t))$ for almost every $t \in \R$ such that $x(\R) \subset L$ and $x(0) = \bar{x}$.
	\label{lem:limitingSet}
\end{lemma}
\begin{proof}
	The assumption that $L(s)$ is bounded by a fixed $M$ for small enough $s$ ensures that all considered sequences are bounded and we may restrict all the asymptotics to happen in a fixed bounded set (for example of size $2M$), in particular, $Z$ will be bounded by $\zeta$ on this set. This ensures that $L$ is nonempty and this that all considered objects remain in a bounded set. 

    Let $\bar{x} \in L$. For a fixed $\alpha > 0$, $ \bar{x} \in \mathrm{cl} \cup_{s \leq \alpha} L(s)$ means that for any $e > 0$, there is $s \leq \alpha$ such that $\dist(\bar{x}, L(s)) \leq e$. Therefore if $\bar{x} \in L$, we deduce that there exists a sequence $(s_j)_{j\in\N}$ which tends to $0$ and a sequence $\bar{x}_j \in L(s_j)$ which tends to $\bar{x}$ as $j \to \infty$.
    
	Fix an arbitrary $T > 0$ and $j \in \N$. We set $K_j = \left\lceil T / s_j \right\rceil$ and consider for $k \geq K_j$, $X_{j,k} = (x_i(s_j))_{i = k - K_j}^{i = k +K_j }$. Up to a subsequence, as $k \to \infty$, $X_{j,k}$ converges to $\bar{X}_j$ which contains $2K_j +1$ accumulation points in $L(s_j)$, the central one chosen to be $\bar{x}_j$. The affine interpolation of the (ordered) collection of points in $\bar{X}_j$ is denoted  $\bar{\gamma}_j \colon [-T,T] \to \R^p$ with $\bar{\gamma}_j(0) = \bar{x}_j$. 

	This construction can be repeated for all $j \in \N$ and using Arzelà-Ascoli, there is a subsequence of $(\bar{\gamma}_j)_{j \in \N}$ which converges uniformly to $\bar{\gamma} \colon [-T,T] \to \R^p$. We have $\bar{\gamma}([-T,T]) \subset L$. Indeed let $t \in [-T, T]$. For any $e > 0$, we have $\|\bar{\gamma}_j(t) - \bar{\gamma}(t)\| \leq e$ for $j$ high enough. By construction, $\bar{\gamma}_j(s_j \lceil t/s_j \rceil) \in L(s_j)$. Since $s_j \lceil t/s_j \rceil$ converges to $t$ and $\bar{\gamma}_j$ is Lipschitz with same constant for all $j$,  then $\|\bar{\gamma}(t) - \bar{\gamma}_j(s_j \lceil t/s_j \rceil) \| \leq 2e$ for all $j$ high enough. Since $e$ was arbitrary, this means $\bar{\gamma}(t) \in L$.
 
    We claim that $\bar{\gamma}$ is a solution to the differential inclusion $\frac{d}{dt} \bar{\gamma}(t) \in Z(\bar{\gamma}(t))$ for almost all $t$.

	Indeed, for any $j \in \N$, there is $k_j \in \N$ such that $\|X_{j,k_j} - \bar{X}_j\| \leq 1/j$. Let $\gamma_j \colon [-T,T] \to \R^p$ be the affine interpolation of $X_{j,k_j}$. Up to the Arzelà-Ascoli subsequence, $\gamma_j \to \bar{\gamma}$ uniformly on $[-T,T]$. Furthermore, for each $j$ 
	\begin{align*}
		\int_{-T}^{T}\dist((\gamma_j(t),\dot{\gamma}_j(t)),\operatorname{graph}[Z])\,dt \leq 2Ts_j \zeta
	\end{align*}
	By \Cref{lem:discreteContinuous}, this shows that $\bar{\gamma}$ is a solution as required.

	The function $\bar{\gamma}$ may be extended to $[-2T, 2T]$, $[-3T, 3T]$, \ldots by taking further subsequences for each $j$ and further Arzelà-Ascoli subsequences. The resulting function is defined on $\R$ and has to be a solution to the differential inclusion. This concludes the proof. 
\end{proof}

The main result stated in \Cref{th:constantStepSizeSemiAlgebraic} is a consequence of the following result and \Cref{lem:SA}.

\begin{theorem}[Convergence for biased subgradient with constant steps]
Under Assumptions \ref{ass:KLmetricReg}, let $C, \rho$ be given by \Cref{lem:estimateSolutionToCrit} and $\epsilon < \bar{\epsilon}$. Set for all $\alpha>0$,  $(x_k(\alpha))_{k \in \N}$, as given by \eqref{eq:subgradVanishingStepSize} with  $\alpha_k = \alpha$, for all $k \in \N$. Then if $\limsup_{\alpha \to 0} \sup_{k \to \infty} \|x_k(\alpha)\| < \infty$,
\begin{align*}
				\lim_{\alpha \to 0} \limsup_{k \to \infty} \dist (f(x_k(\alpha)), \vcrite f ) &=0,\\
				\limsup_{\alpha \to 0} \limsup_{k \to \infty} \dist (x_k(\alpha), \crit f ) &\leq C \epsilon^{\rho}. 
\end{align*}
The boundedness condition always holds if $f$ is coercive.
\label{th:constantStepSize}
\end{theorem}

\begin{proof}
	By assumption, there is $M> 0$ such that for small enough $\alpha$, $\|x_k(\alpha)\| \leq M$ for all $k \in \N$ large enough.

	For each $s$ denote by $L(s)$ the set of limit points of $x_k(s)$, which is closed. 
	We set  $L = \cap_{\alpha > 0} \mathrm{cl} \cup_{s\leq \alpha} L(s)$ which is closed as an intersection of closed set, and bounded by $M$. Intuitively $L$ is the set of accumulation points of accumulation points of $(x_k(\alpha))_{k \in \N}$ as $\alpha \to 0$. By \Cref{lem:limitingSet}, $L$ is invariant. We will show that $f(L) \subset \vcrite f$ from which the first statement follows. The second is then a consequence of \Cref{cor:invariantSet}.

	For any $l \notin \vcrite f$, since $\vcrite f$ is closed by \Cref{ass:KL-MR}, \Cref{lem:descentSequence2} ensures that there is $\bar{\alpha}$ and $\eta$, such that for any $s \leq \bar{\alpha}$, $[l-2\eta,l+2\eta] \cap  f(L(s)) = \emptyset$. We deduce that  
	\begin{align*}
			[l-2\eta,l+2\eta] \cap  f\left(\cup_{s \leq \bar{\alpha}} L(s)\right) &= \emptyset \\ 
			[l-\eta,l+\eta] \cap  f\left(\mathrm{cl} \cup_{s \leq \bar{\alpha}} L(s)\right) &= \emptyset \:\:\mbox{ (use continuity).}\\ 
	\end{align*}
	Whence we see that for any $l \not \in \vcrite f$, $l \not \in f\left(\cap_{\alpha > 0}\mathrm{cl} \cup_{s \leq \alpha } L(s)\right)$, and therefore $f(L) \subset \vcrite f$ through complementation. This proves the first claim. 

    The last comment on boundedness follows from \Cref{lem:descentSequence}.
\end{proof}

\section{The convex case}
\label{sec:convex}

\begin{proof}[of \Cref{th:convexSemiAlgebraic}]
	Consider a sequence of step sizes $(\alpha_k)_{k \in \N}$ and a iterates $(x_k)_{k \in \N}$ such that for all $k \in \N$,
	\begin{align*}
	    x_{k+1} = x_k - \alpha_k (v_k + b_k),
	\end{align*}
	for some $v_k \in \partialc f(x_k)$ and $\|b_k\| \leq \epsilon$.
	We have for any $k \in \N$ and any $x^* \in X^*$,
	\begin{align*}
	    \frac{1}{2}\|x_{k+1} - x^*\|_2^2 &= \frac{1}{2}\|x_k - \alpha_k (v_k+b_k) - x^*\|_2^2\\ 
	    &=\frac{1}{2}\|x_k - x^*\|_2^2 + \alpha_k (v_k+b_k)^T (x^* - x_k ) + \frac{\alpha_k^2}{2} \|v_k + b_k\|_2^2 \\
	    &\leq\frac{1}{2}\|x_k - x^*\|_2^2 - \alpha_k (f(x_k) - f^*) + \alpha_k \epsilon \|x_k - x^*\|  \\ & +  \frac{\alpha_k^2}{2} (L+\epsilon)^2. 
	\end{align*}
	We deduce, using the error bound \eqref{eq:errorbound}, that
	\begin{align}
            \label{eq:mainIneqConvex}
	    & \frac{1}{2}\dist(x_{k+1}, X^*)^2 \\
            \leq\;&\frac{1}{2}\dist(x_{k}, X^*)^2 - \alpha_k (f(x_k) - f^*) + \alpha_k \epsilon \dist(x_{k}, X^*) +  \frac{\alpha_k^2}{2} (L+\epsilon)^2 \nonumber\\
	    \leq\;& \frac{1}{2}\dist(x_{k}, X^*)^2 + \alpha_k \left(\epsilon \frac{c}{2} \left((f(x_k) - f^*)^a + (f(x_k) - f^*)\right) -  (f(x_k) - f^*)\right)\nonumber \\ & +\frac{\alpha_k^2}{2} (L+\epsilon)^2. \nonumber
	\end{align}
	\paragraph{If $a=1$:}
	This means that we have a sharp function, then, as long as $\epsilon c < 1$, we obtain the same global rate as the classical subgradient algorithm, modulo a constant $(1 - \epsilon c)$, and all accumulation points in the argmin. Indeed, summing the previous inequality from $i = 0$ to $k$ leads to
    \begin{equation*}
        (1 - \epsilon c) \frac{\sum_{i=0}^k \alpha_i (f(x_i) - f^*)}{\sum_{i=0}^k \alpha_i} \leq \frac{\|x_0 - x^{*}\|^2_2 + (L + \epsilon)^2 \sum_{i = 0}^{k} \alpha_i^2}{\sum_{i = 0}^{k} \alpha_i}.
    \end{equation*}
    We remark that this is exactly the claimed formula for $a = 1$.
\paragraph{If $a<1$:}
Then we can use the following lemma

\begin{lemma}
    Set $g(\delta) = s \delta^t - \delta$ for some $t \in (0,1)$ and $s> 0$, then for all $\delta \in \RR_+$
    \begin{align*}
        g(\delta) \leq - (1 - t) (\delta - s^{\frac{1}{1-t}}).
    \end{align*}
    \label{lem:numericLemma}
\end{lemma}
\begin{proof}
    The function $g$ is concave on $\RR_+$, set $\delta_0 = s^{\frac{1}{1 - t}}$, we have $g(\delta_0) = 0$ and $g'(\delta_0) = t - 1 < 0$  and the result follows by concavity.
\end{proof}

We obtain setting $\delta = f(x_k) - f^*$ applying \Cref{lem:numericLemma} to \eqref{eq:mainIneqConvex} 
	\begin{align*}
	    &\frac{1}{2}\dist(x_{k+1}, X^*)^2 -  \frac{1}{2}\dist(x_{k}, X^*)^2 \\
            \leq\;&   \frac{\alpha_k}{2} \left(-(1-a) \left( (f(x_k) - f^*) -  (\epsilon c)^{\frac{1}{1-a}}\right) + (\epsilon c-1) (f(x_k) - f^*) \right) \\ & +  \frac{\alpha_k^2}{2} (L+\epsilon)^2\\
            = \; & \frac{\alpha_k}{2} \left((\epsilon c + a - 2) (f(x_k) - f^*) + (1-a) (\epsilon c)^{\frac{1}{1-a}}\right) +  \frac{\alpha_k^2}{2} (L+\epsilon)^2
	\end{align*}
         We deduce by summation
	\begin{align*}
	   \left( 2 - a - \epsilon c \right)\frac{\sum_{i=0}^k \alpha_i (f(x_i) - f^*)}{\sum_{i=0}^k \alpha_i} & \leq  (1 - a)(\epsilon c)^{\frac{1}{1-a}} \\ & + \frac{\|x_0 - x_*\|^2 + (L+\epsilon)^2\sum_{i=0}^k  \alpha_i^2 }{\sum_{i=0}^k \alpha_i}.
	\end{align*}
\end{proof}

\section*{Acknowledgements}
The authors acknowledge the support of Centre Lagrange.
 JB, TL, EP thank AI Interdisciplinary Institute ANITI funding, through the French ``Investments for the Future -- PIA3'' program under the grant agreement ANR-19-PI3A0004, Air Force Office of Scientific Research, Air Force Material Command, USAF, under grant numbers FA8655-22-1-7012, ANR Chess, grant ANR-17-EURE-0010, ANR Regulia.

%
%

\bibliographystyle{spmpsci}      
\bibliography{references}   

\end{document}